\documentclass{amsart}
\usepackage{amsfonts}
\usepackage{amsmath, amsthm, amsfonts, amssymb, color}
 \usepackage{mathrsfs}
\usepackage{amsfonts, amsmath}

\usepackage[dvipdfm, 
            pdfstartview=FitH,
            CJKbookmarks=true,
            bookmarksnumbered=true,
            bookmarksopen=true,
            colorlinks, 
            pdfborder=001,  
            linkcolor=red,
            anchorcolor=red,
            citecolor=blue
            ]{hyperref}

 \usepackage{pgf,tikz}

\begin{document}
\addtolength{\parskip}{8pt}
\parindent15pt

\newcommand\C{{\mathbb C}}

\newtheorem{thm}{Theorem}[section]
\newtheorem{prop}[thm]{Proposition}
\newtheorem{cor}[thm]{Corollary}
\newtheorem{lem}[thm]{Lemma}
\newtheorem{lemma}[thm]{Lemma}
\newtheorem{exams}[thm]{Examples}
\theoremstyle{definition}
\newtheorem{defn}[thm]{Definition}
\newtheorem{rem}[thm]{Remark}
\newcommand\RR{\mathbb{R}}
\newcommand{\la}{\lambda}
\def\RN {\mathbb{R}^n}
\newcommand{\norm}[1]{\left\Vert#1\right\Vert}
\newcommand{\abs}[1]{\left\vert#1\right\vert}
\newcommand{\set}[1]{\left\{#1\right\}}
\newcommand{\Real}{\mathbb{R}}
\newcommand{\supp}{\operatorname{supp}}
\newcommand{\card}{\operatorname{card}}
\renewcommand{\L}{\mathcal{L}}
\renewcommand{\P}{\mathcal{P}}
\newcommand{\T}{\mathcal{T}}
\newcommand{\A}{\mathbb{A}}
\newcommand{\K}{\mathcal{K}}
\renewcommand{\S}{\mathcal{S}}
\newcommand{\blue}[1]{\textcolor{blue}{#1}}
\newcommand{\red}[1]{\textcolor{red}{#1}}
\newcommand{\Id}{\operatorname{I}}
\newcommand\wrt{\,{\rm d}}
\def\SH{\sqrt {H}}

\newcommand{\rn}{\mathbb R^n}
\newcommand{\de}{\delta}
\newcommand{\tf}{\tfrac}
\newcommand{\ep}{\epsilon}
\newcommand{\vp}{\varphi}

\newcommand{\mar}[1]{{\marginpar{\sffamily{\scriptsize
        #1}}}}
\newcommand{\li}[1]{{\mar{LY:#1}}}
\newcommand{\el}[1]{{\mar{EM:#1}}}
\newcommand{\as}[1]{{\mar{AS:#1}}}
\newcommand\CC{\mathbb{C}}
\newcommand\NN{\mathbb{N}}
\newcommand\ZZ{\mathbb{Z}}
\renewcommand\Re{\operatorname{Re}}
\renewcommand\Im{\operatorname{Im}}
\newcommand{\mc}{\mathcal}
\newcommand\D{\mathcal{D}}
\newcommand{\al}{\alpha}
\newcommand{\nf}{\infty}
\newcommand{\comment}[1]{\vskip.3cm
	\fbox{%
		\color{red}
		\parbox{0.93\linewidth}{\footnotesize #1}}
	\vskip.3cm}

\newcommand{\disappear}[1]

\numberwithin{equation}{section}
\newcommand{\chg}[1]{{\color{red}{#1}}}
\newcommand{\note}[1]{{\color{green}{#1}}}
\newcommand{\later}[1]{{\color{blue}{#1}}}
\newcommand{\bchi}{ {\chi}}

\numberwithin{equation}{section}
\newcommand\relphantom[1]{\mathrel{\phantom{#1}}}
\newcommand\ve{\varepsilon}  \newcommand\tve{t_{\varepsilon}}
\newcommand\vf{\varphi}      \newcommand\yvf{y_{\varphi}}
\newcommand\bfE{\mathbf{E}}
\newcommand{\ale}{\text{a.e. }}

 \newcommand{\mm}{\mathbf m}
\newcommand{\Be}{\begin{equation}}
\newcommand{\Ee}{\end{equation}}

\title[Hardy's inequality for  Hermite   expansions  revisited  ]
{  Hardy's inequality for  Hermite   expansions revisited }

 \author{Peng Chen}
  \author{Jinsen Xiao}
 \address{Peng Chen, Department of Mathematics, Sun Yat-sen
 University, Guangzhou, 510275, P.R. China}
 \email{chenpeng3@mail.sysu.edu.cn}
 \address{Jinsen Xiao, School of Science, Guangdong University of Petrochemical Technology, Maoming 525000, P.R. China}
\email{jinsenxiao@gdupt.edu.cn}

\date{\today}

\begin{abstract} In this article,  we give a short proof of Hardy's inequality  for  Hermite expansions
of functions in the classical Hardy spaces  $H^p({\mathbb R^n})$,
 by using an atomic decomposition of the Hardy spaces associated with the Hermite operators.
When the space dimension is $1$, we obtain a new estimate of Hardy's inequality for  Hermite   expansions
in $H^p({\mathbb R})$ for the range $0<p<1.$

\end{abstract}

\maketitle

\section{\bf Introduction}
  \setcounter{equation}{0}

A function  $f(z)$  analytic in the unit disk  $|z|<1$ for $z\in \mathbb{C}$ is  said
to  be  of  class  $H^p(\mathbb{C})$ where $0<p<\infty$  if
$$\lim_{r\to 1}\left\{\frac{1}{2\pi}\int_0^{2\pi}|f(re^{i\theta})|^p d\theta\right\}^{\frac{1}{p}}<\infty.$$
For $f(z)=\sum_{k=0}^{\infty} a_k z^k,$ it is interesting to find a condition on the Taylor coefficient
 $a_k$ which is  both necessary and sufficient for  $f$ to  be in  $H^p(\mathbb{C})$
 for some  $p\in (0,\infty).$  If  $1<p<\infty$,  the problem  is  equivalent to  that  of  describing the
Fourier coefficients  of  $L^p$  functions,  as  the  M.  Riesz theorem shows.
For the case $0<p\leq 1$, the known results can be described as  Hardy's inequality given by
 \begin{eqnarray}
\sum_{k=0}^{\infty}{|a_k|^p\over (k+1)^{2-p}}\leq C_p\|f\|_{H^p(\mathbb{C})}^p,\label{a}
 \end{eqnarray}
where the constant $C_p$ depends only on $p$ (see  \cite[Theorems 6.2]{duren}).

Analogues of Hardy's inequality in the context of eigenfunction expansions have
been considered, see \cite{CT, kanjin, kanjin1, li2014, radha} and the references therein.
Recall that  the Hermite operator $L$  on $ \RR^n$ is defined by
   \begin{eqnarray}\label{eecc}
L=-\Delta + |x|^2 =-\sum_{i=1}^n {\partial^2\over \partial x_i^2} + |x|^2, \quad x=(x_1, \cdots, x_n)\in\RR^n.
\end{eqnarray}
The operator $L$ is non-negative and self-adjoint with respect to the Lebesgue measure
on $\RN$. For  $k\in \mathbb N =\{0,1,2,\cdots\},$ the Hermite polynomials $H_k(t) $ on $\RR$ are
defined by $H_k(t)=(-1)^k e^{t^2} {d^k\over d t^k} \big(e^{-t^2}\big)$, and  the Hermite functions
$h_k(t):=(2^k k !  \sqrt{\pi})^{-1/2} H_k(t) e^{-t^2/2}$, $k=0, 1, 2, \cdots$ form an orthonormal basis
of $L^2(\mathbb R)$.
For
any multiindex $\mu\in {\mathbb N}^n$,
 the $n$-dimensional Hermite functions are given by tensor product of the one dimensional Hermite functions:
\begin{eqnarray}\label{ephi}
\Phi_{\mu}(x)=\prod_{i=1}^n h_{\mu_i}(x_i), \quad \mu=(\mu_1, \cdots, \mu_n).
\end{eqnarray}
Then the functions $\Phi_{\mu}$ are eigenfunctions for
the Hermite operator with eigenvalues $(2|\mu|+n)$ and $\{\Phi_{\mu}\}_{\mu\in \mathbb N^n}$ forms a complete orthonormal
system in $L^2({\RN})$ (see \cite{Th4}).
Thus, for  every $f\in L^2(\RN)$  we have  the Hermite expansion
\begin{eqnarray} \label{e1.3}
f(x)=\sum_{\mu}\langle f, \Phi_{\mu}\rangle \Phi_\mu(x)\ \ \ {\rm with} \ \  \ \|f\|^2_{L^2(\RN)}= \sum_{\mu}|\langle f, \Phi_{\mu}\rangle|^2.
\end{eqnarray}

Kanjin \cite{kanjin} obtained  the following Hardy's inequality in the context of one-dimensional
 Hermite functions, namely,
 \begin{eqnarray}
  \sum\limits_{k=0}^{\infty}  \frac{|\langle f, {h}_k\rangle |}{(k+1)^{\frac{29}{36}}}
  \leq C \|f\|_{H^1(\mathbb{R})}.\label{b}
 \end{eqnarray}
 Radha and Thangavelu \cite{radha}
proved inequalities of Hardy's type for $n$-dimensional Hermite expansions for $n\geq 2$ and $0<p\leq1,$
 \begin{eqnarray}
\sum\limits_{\mu\in \mathbb{N}^n} {|\langle f, \Phi_{\mu} \rangle |^p\over (2|\mu|+n)^{\frac{3n}{4}(2-p)}}\leq C\|f\|^p_{H^{p}(\mathbb{R}^n)}.
\label{cc}
 \end{eqnarray}
   However, their method  does not work for the one-dimensional case.   Kanjin in  \cite{kanjin1}
obtained an improved form of (\ref{b}) with $\frac{3}{4}+\epsilon$ for $\epsilon>0$ in place of $\frac{29}{36}$, and 
  conjectured that the possible  form  should be
 \begin{eqnarray}
  \sum\limits_{k=0}^{\infty}   \frac{|\langle f, {h}_k\rangle |}{(k+1)^{\frac{3}{4}}}\leq C \|f\|_{H^1(\mathbb{R})}. \label{c}
 \end{eqnarray}
  Li, Yu and Shi \cite{li2014} gave a positive answer to prove estimate \eqref{c}  by using a  different approach to  evaluate the square
  integration of the Poisson integral associated to Hermite
expansions of functions in $H^1(\mathbb{R}).$

 The aim of this paper is to prove  Hardy's inequality  for  Hermite expansions
of functions
in a class Hardy spaces $H^p_L({\mathbb R^n})$
associated with the operator
$L=-\Delta+|x|^2$ for $n\geq1$ and  $0<p\leq 1.$
  Our  main result is stated as follows.
\begin{thm}\label{th1.1}
Let  $n\geq 1$ and  $0<p\leq 1.$ Then there exists a constant $C>0$ such that for $f\in H^{p}_{L}(\mathbb{R}^n),$
 \begin{eqnarray}
\sum\limits_{\mu\in \mathbb{N}^n} {|\langle f, \Phi_{\mu} \rangle |^p\over (2|\mu|+n)^{\frac{3n}{4}(2-p)}}\leq C\|f\|^p_{H_{L}^{p}(\mathbb{R}^n)}.
 \label{11c}
 \end{eqnarray}
As a consequence, we have that for any  $f\in H^{p}(\mathbb{R}^n)$,
 \begin{eqnarray}
\sum\limits_{\mu\in \mathbb{N}^n} {|\langle f, \Phi_{\mu} \rangle |^p\over (2|\mu|+n)^{\frac{3n}{4}(2-p)}}\leq C\|f\|^p_{H^{p}(\mathbb{R}^n)}.
 \label{11d}
 \end{eqnarray}
\end{thm}

We would like to  mention that  in the proof of Theorem \ref{th1.1},
we do not need to estimate
 the Taylor expansion of the function $\sum_{|\mu|=k} \Phi_{\mu}(x)  \Phi_{\mu} (y)$ as in \cite{CT, kanjin, kanjin1, li2014, radha}.
 Instead,  our proof heavily depends on the   atomic decomposition   of the
Hardy spaces   $H^p_L(\RN)$   in \cite{Hofmann2011}, see Section 3 below.
Besides, when the space dimension is $1$,  we obtain a new estimate of Hardy's inequality for  Hermite   expansions
in $H^p({\mathbb R})$ for the range $0<p<1.$

The paper is organized as follows. In Section 2, we recall the definition of the
Hardy spaces   $H^p_L(\RN)$ associated with $L=-\Delta+|x|^2$ and its    atomic decomposition,
and show that the classical Hardy space $H^p(\RN), 0<p\leq 1,$ is a proper subspace of
the space  $H^p_L(\RN)$ associated with $L$.
 Our main result, Theorem \ref{th1.1}, will be proved in Section 3.

Throughout, the letters $C$ and $c$ will denote (possibly different) constants that are
independent of the essential variables.

\section{\bf The space  $H^p(\RN)$ is a proper subspace of   $H_L^p(\RN)$ for $0<p\leq 1$
 }
 \setcounter{equation}{0}

 Recall that  Hardy space $H^p({\mathbb{R}}^n)$ can be defined   in terms of the maximal function associated with
   the heat semigroup generated by the Laplace operator $\Delta$ on   $\mathbb R^n$.
 Following \cite{St2}, a distribution
 $f$   is said to be in  $H^p({\mathbb{R}}^n), 0<p\leq 1$  if
\begin{eqnarray}\label{e1.1a}
{\mathcal M}_{\Delta} f(x)=\sup_{ t>0}\abs{e^{{t}\Delta}f(x)}
\end{eqnarray}
belongs to $  L^p({\mathbb{R}}^n)$. If this is the case, then we set
$
\|f\|_{H^p({\mathbb{R}}^n)}=\|{\mathcal M}_{\Delta} f\|_{L^p({\mathbb{R}}^n)}.
$

The  Hardy space associated with the Hermite operator $L=-\Delta + |x|^2$
has attracted much attentions in the last decades and has been
 a very active research topic in harmonic analysis, see, for example, \cite{Dziubanski, Hofmann2011}.
Following \cite{Dziubanski}, we say that $f$ is in the space $H^p_{L}({\mathbb R}^n)$ if
$$
{\mathcal M}_Lf(x):=\sup_{t>0} \big|e^{-tL} f(x)\big|
$$
is in $L^p({\mathbb R}^n).$ The quasi-norm in $H^p_{L}$ is defined by
$$
\| f \|_{H^p_{L}({\mathbb R}^n)}^p=\|{\mathcal M}_Lf\|_{L^p}^p.
$$
When $p>1,$ $H^p_{L}({\mathbb R}^n)\simeq L^p(\mathbb{R}^n).$

Now, define  an auxilliary function $m(x)$ by
$$
m(x)=\sum_{\beta\leq (2,  \cdots, 2) } \big|D^{\beta} |x|^2  \big|^{-(|\beta|+2)},
$$
where $|\beta|=|(\beta_1, \cdots, \beta_n)|=\sum_{i=1}^n \beta_i.$ There exists a constant $c>0$ such that
$m(x)>c$ for every $x\in \RR^n.$ Set
$$
{\mathcal B}_0=\{ x|x\in \RR^n, c\leq m(x)\leq 1\};
$$
$$
{\mathcal B}_k=\{ x|x\in \RR^n, 2^{k-1\over 2}\leq m(x)\leq 2^{k\over 2}\}, k=1,2,3, \cdots.
$$
We say that a function $a$ is a $p$-atom for the space $H^p_{L}({\mathbb R}^n)$ associated to a ball $B(x_0, r)=\{x\in\mathbb{R}^n: |x-x_0|<r\}$ if
\begin{itemize}
\item[(i)]  $\supp a\subseteq B(x_0, r)$;\\
\item[(ii)] $\|a\|_{L^{\infty}}\leq |B(x_0, r)|^{-1/p}$;\\
\item[(iii)] If $x_0\in {\mathcal B}_k, $ then $r\leq 2^{1-{k\over 2}};$\\
\item[(iv)] If $x_0\in {\mathcal B}_k  $  and  $r\leq 2^{-1-{k\over 2}},$ then $\int x^{\beta} a(x)dx=0 $ for all $|\beta|\leq n({1\over p}-1).$
\end{itemize}

In \cite[Theorem 1.12]{Dziubanski}, Dziuba\'{n}ski  obtained the following atomic characterization of  $H^p_{L}({\mathbb R}^n).$
\begin{prop}\label{prop2.1}
 A distribution $f$ is in $H^p_{L}({\mathbb R}^n), 0<p\leq 1$ if and only if there exist  $\lambda_j\in {\mathbb R}$ and
$p$-atom $a_j, j=0,1, 2, \cdots,$ such that
$$
f(x)=\sum_{j=0}^{\infty} \lambda_j a_j(x)
$$
and
$$
C_1\|f\|_{H^p_L}^p \leq \sum_{j=0}^{\infty} |\lambda_j|^p \leq  C_2\|f\|_{H^p_L}^p,
$$
where   constants $C_1, C_2$ depend  only on   $ p$.
\end{prop}

It follows from Proposition~\ref{prop2.1} and the properties of ${\mathcal B}_k$ that
  $H^p({\mathbb R}^n)$ is a proper subspace of the space $H^p_{L}({\mathbb R}^n)$ (see \cite[p.77]{Dziubanski}).
   We can decompose every element in $H^p_{L}({\mathbb R}^n) $ into atoms that are supported on small balls, but some atoms may not have the moment
   condition. That is,

   \begin{prop}\label{prop2.2}
Let $n\geq 1$ and $0<p\leq 1.$ Let $L=-\Delta +|x|^2$. Then we have that
$$
H^p({\mathbb R}^n)\subsetneqq  H^p_{L}({\mathbb R}^n).
$$
\end{prop}

\section{\bf Proof of Theorem \ref{th1.1}}
 \setcounter{equation}{0}

 To show Theorem \ref{th1.1}, we need  an $L$-atomic decomposition of $H^p_{L}({\mathbb R}^n)$  (see \cite{Hofmann2011, SY}).
 Let $n\geq 1$ and $0<p\leq 1$. Assume that $M$ is an integer satisfying $$M>\frac{n(2-p)}{4p}.$$
Let $\mathcal {D}(T)$ be the domain of an operator $T$ and $B=B(x_B,r_B)$ with the measure $V(B)=cr_B^{n}.$

Given   $0<p\leq 1,$ a function  $a\in L^2(\mathbb{R}^n)$ is called a $(p, L, M)-$atom associated with the operator $L$
if there exists a function $b\in \mathcal {D}(L^M)$ and a ball $B\subset\mathbb{R}^n$ such that
\begin{itemize}
\item[(1)]  $a=L^M b;$\\
\item[(2)]  $\text{supp}\, L^k b\subset B,\ k=0,1,\cdots,M;$\\
\item[(3)]$\|(r^2_{B}L)^k b\|_{L^2 (\mathbb{R}^n)}\leq r^{2M}_{B} V(B)^{1/2-1/p},\ k=0,1,\cdots,M.$
\end{itemize}

Following \cite{SY}, for a function $f\in L^2(\mathbb{R}^n),$ we will say that $f=\sum_{j=0}^{\infty} \lambda_j a_j$
is an atomic $(p, L, M)-$representation if $\{\lambda_j\}_{j=0}^{\infty}\in\ell^p,$ each $a_j$ is a $(p, L, M)-$atom,
and the sum converges in $L^2(\mathbb{R}^n).$ Set
$$
\mathbb{H}^p_{L,at,M}(\mathbb{R}^n):=\{f\in L^2(\mathbb{R}^n): f \ \text{has an atomic }(p, L, M)\text{-representation}\},
$$
with the norm $\|f\|_{\mathbb{H}_{L,at,M}(\mathbb{R}^n)}^p$ given by
 \begin{eqnarray*}
\inf\bigg\{ \bigg(\sum\limits_{j=0}^{\infty} |\lambda_j|^p\bigg)^{\frac{1}{p}}: f=\sum\limits_{j=0}^{\infty}
 \lambda_j a_j\hbox{ is an atomic }(p, L, M)\hbox{-representation}\bigg\}.
 \end{eqnarray*}
 The atomic Hardy space $H^p_{L,at,M}(\mathbb{R}^n)$ is then defined as the completion of
  $\mathbb{H}^p_{L,at,M}(\mathbb{R}^n)$ with respect to this norm. Then we have the following result.

\begin{prop}\label{th3}
For $0<p\leq 1,$  we have
$$H^p_{L}(\mathbb{R}^n)\simeq H^p_{L,at,M}(\mathbb{R}^n).$$
Moreover, for $f\in H^p_{L}(\mathbb{R}^n)$, there exist $(p, L, M)$-atoms $\{a_j\}_{j=0}^{\infty}$ and $\{\lambda_j\}_{j=0}^{\infty}\in\ell^p,$ such that $f=\sum_{j=0}^{\infty}\lambda_j a_j $ is in $H^p_{L}$ and
$$
\sum_{j=0}^{\infty}|\lambda_j|^p\leq C\|f\|_{H^p_{L}}^p.
$$
\end{prop}
\begin{proof}
 For the proof of the first conclusion,  we refer the reader to  \cite[Theorem 8.2]{Hofmann2011} for $p=1$ and \cite[Theorem 1.3]{SY} for general $0<p\leq 1$.
 For the proof of the second conclusion,  we refer the reader to~\cite[Corollary 4.1]{JY}.
 \end{proof}

Let $\phi=L^M\nu$ be a function in $L^2(\RN)$, where $\nu\in\mathcal{D}(L^M)$. Denote $U_0(B)=B(0,1)$ and $U_j(B)=B(0,2^j)\backslash B(0,2^{j-1})$ for $j=1,2,\cdots$. Following \cite{Hofmann2011, HM,JY}, for $\epsilon>0$ and $M\in\mathbb{N}$, we introduce the space
$$\mathcal{M}^{M,\epsilon}(L):=\big\{\phi=L^M\nu\in L^2(\RN):
\ \|\phi\|_{\mathcal{M}^{M,\epsilon}(L)}<\infty\big\},$$
where
$$\|\phi\|_{\mathcal{M}^{M,\epsilon}(L)}:= \sup_{j\in\mathbb{N}}
\left\{2^{j\epsilon}2^{nj(1/2+1/p-1)}\sum_{k=0}^M\|L^k\nu\|_{L^2(U_j(B))}\right\}.$$
Then for any $M\in \mathbb{N}$, define
$$\mathcal{M}^{M}(L):= \bigcap_{\epsilon>0}
(\mathcal{M}^{M,\epsilon}(L))^\ast.$$
 A functional  $f\in\mathcal{M}^{M}(L)$ is said to be in ${\rm Lip}_{p,L}^M(\RN)$ if
\begin{equation*}
\|f\|_{{\rm Lip}_{p,L}^M}:=\sup_{B\subset\RN}V(B)^{1-1/p}\left[\frac{1}{V(B)}\int_B
|(I-e^{-r_B^2L})^Mf(x)|^2 \,dx\right]^{1/2}< \infty,
\end{equation*}
where the supremum is taken over all ball $B$ of $\RN$. Then by~\cite[Theorem 4.1]{JY}, we have the following dual result.
\begin{prop}\label{dual}
For $0<p\leq 1,$  we have
$$
\left(H^p_{L}(\mathbb{R}^n)\right)^*= {\rm Lip}_{p,L}^M(\mathbb{R}^n).$$

\end{prop}

Now we are ready to prove   Theorem~\ref{th1.1}.

\begin{proof}[ \bf{Proof of  Theorem~\ref{th1.1}}]
 Let  $n\geq 1$ and $0<p\leq 1$.
Set
$$
\sigma(n,p)=\frac{3n}{4}\big(2-p\big).
$$
First, we claim that that there exists a constant $C>0$ such that  for any $(p, L, M)$-atom $a$, supported in $B(x_B,r)$,
\begin{eqnarray}\label{foratom}
 \sum\limits_{\mu\in \mathbb{N}^n} {|\langle a, \Phi_{\mu} \rangle |^p \over (2|\mu|+n)^{ \sigma(n,p)}} \leq C <\infty.
\end{eqnarray}

We decompose the summation over $\mu$ by
\begin{eqnarray*}
  \sum\limits_{\mu\in \mathbb{N}^n} {|\langle a, \Phi_{\mu} \rangle |^p \over (2|\mu|+n)^{ \sigma(n,p)}}
&=&\sum_{j\geq -1} \ \sum_{2^{j}<2|\mu|+n\leq 2^{j+1}} {|\langle a, \Phi_{\mu} \rangle |^p \over (2|\mu|+n)^{ \sigma(n,p)}}\\
&=& \sum_{j:  \, 2^jr^2\geq 1} \ \sum_{2^{j}<2|\mu|+n\leq 2^{j+1}}{|\langle a, \Phi_{\mu} \rangle |^p \over (2|\mu|+n)^{ \sigma(n,p)}}
\\
&\ \ \ \ \  +&   \sum_{j: \ 2^jr^2< 1} \ \sum_{2^{j}<2|\mu|+n\leq 2^{j+1}}{|\langle a, \Phi_{\mu} \rangle |^p \over (2|\mu|+n)^{ \sigma(n,p)}}
\\
&=&{\rm I}+{\rm II}.
\end{eqnarray*}

For the term ${\rm I}$, we apply H\"older's inequality, Plancherel type equality~\eqref{e1.3} and estiamte~(3) in the definition of $(p, L, M)$-atom to get
\begin{eqnarray*}
{\rm I}&=& \sum_{j:\ 2^jr^2\geq 1}\sum_{2^{j}<2|\mu|+n\leq 2^{j+1}}{|\langle a, \Phi_{\mu} \rangle |^p \over (2|\mu|+n)^{ \sigma(n,p)}}\\
&\leq& \sum_{j:\ 2^jr^2\geq 1} \left(\sum_{2^{j}<2|\mu|+n\leq 2^{j+1}} |\langle a, \Phi_{\mu} \rangle|^2\right)^{\frac{p}{2}}
   \left(\sum_{2^{j}<2|\mu|+n\leq 2^{j+1}}  (2|\mu|+n)^{-\frac{2\sigma(n,p)}{2-p}}\right)^{\frac{2-p}{2}}\\
&\leq& C \sum_{j:\ 2^jr^2\geq 1}\|a\|_2^p     2^{-j(\sigma(n,p)-\frac{n(2-p)}{2})}\\
&=&C r^{\frac{n(p-2)}{2}}\sum_{j:\ 2^jr^2\geq 1}     2^{-\frac{jn(2-p)}{4}}\\
&\leq& C.
\end{eqnarray*}

For the term ${\rm II}$,  we notice that   $ a=L^M b$ and  obtain
$$\langle a, \Phi_{\mu} \rangle=\langle L^M b, \Phi_{\mu} \rangle =\langle b, L^M \Phi_{\mu} \rangle
 =(2|\mu|+n)^{M}\langle b, \Phi_{\mu} \rangle.
 $$
By H\"older's inequality, Plancherel type equality~\eqref{e1.3} and estiamte~(3) in the definition of $(p, L, M)$-atom,
\begin{eqnarray*}
{\rm II}&=&\sum_{ j:\   2^jr^2< 1  } \ \sum_{2^{j}<2|\mu|+n\leq 2^{j+1}}{|\langle a, \Phi_{\mu} \rangle |^p \over (2|\mu|+n)^{ \sigma(n,p)}}\\
&=&\sum_{j:\  2^jr^2< 1} \ \sum_{2^{j}<2|\mu|+n\leq 2^{j+1}} { |\langle b, \Phi_{\mu} \rangle|^p \over (2|\mu|+n)^{-Mp+\sigma(n,p)}}\\
&\leq& \sum_{j:\  2^jr^2< 1} \left(\sum_{2^{j}<2|\mu|+n\leq 2^{j+1}} |\langle b, \Phi_{\mu} \rangle|^2\right)^{\frac{p}{2}}
  \left(\sum_{2^{j}<2|\mu|+n\leq 2^{j+1}}   (2|\mu|+n)^{\frac{2(Mp-\sigma(n,p))}{2-p}}\right)^{\frac{2-p}{2}}\\
&\leq& C \|b\|_2^p \sum_{j:\  2^jr^2< 1}    2^{j(Mp-\sigma(n,p)+\frac{n(2-p)}{2})}\\
&\leq& C r^{2Mp+\frac{n(p-2)}{2}}\sum_{j:\  2^jr^2< 1}     2^{j(Mp-\frac{n(2-p)}{4})}\\
&\leq& C.
\end{eqnarray*}

Thus we complete the proof of estimate~\eqref{foratom}.
Now for $f\in H^p_{L}(\mathbb{R}^n)$, it follows from~Proposition~\ref{th3} that there exist $(p, L, M)$-atoms $\{a_j\}_{j=0}^{\infty}$ and $\{\lambda_j\}_{j=0}^{\infty}\in\ell^p,$ such that $f=\sum_{j=0}^{\infty}\lambda_j a_j $ in $H^p_{L}$ and
$$
\sum_{j=0}^{\infty}|\lambda_j|^p\leq C\|f\|_{H^p_{L}}^p.
$$
A direct calculation shows that $\Phi_{\mu}\in {\rm Lip}_{p,L}^M(\mathbb{R}^n)$. Thus it follows from Proposition~\ref{dual} that
\begin{align*}
|\langle f, \Phi_{\mu} \rangle|&=|\langle\sum_{j=0}^{\infty}\lambda_j a_j, \Phi_{\mu} \rangle|\\
&\leq \overline {\lim_{N\to \infty}}|\langle\sum_{j=0}^{N}\lambda_j a_j, \Phi_{\mu} \rangle|+\overline {\lim_{N\to \infty}}|\langle\sum_{j=N+1}^{\infty}\lambda_j a_j, \Phi_{\mu} \rangle|\\
&\leq \sum_{j=0}^{\infty}|\lambda_j||\langle a_j, \Phi_{\mu} \rangle|+\overline {\lim_{N\to \infty}}C_\mu\|\sum_{j=N+1}^{\infty}\lambda_j a_j\|_{H_L^p}\\
&\leq \sum_{j=0}^{\infty}|\lambda_j||\langle a_j, \Phi_{\mu} \rangle|.
\end{align*}
This, together with  Minkowski's inequality and \eqref{foratom}, shows that  for $0<p\leq 1$,
\begin{eqnarray*}
\sum\limits_{\mu\in \mathbb{N}^n}{|\langle f, \Phi_{\mu} \rangle|^p \over (2|\mu|+n)^{ \sigma(n,p)}}
&\leq& \sum_{j=0}^{\infty} |\lambda_j|^p \sum\limits_{\mu\in \mathbb{N}^n} {| \langle a_j, \Phi_{\mu} \rangle|^p \over (2|\mu|+n)^{ \sigma(n,p)}}\\
&\leq& C \sum_{j=0}^{\infty} |\lambda_j|^p  \leq  C \|f\|_{H^p_{L}}^p.
\end{eqnarray*}
Consequently, we have obtained estimate \eqref{11c}.
By Proposition~\ref{prop2.2}, we see that
   $H^p({\mathbb R}^n)$ is a proper subspace of the space $H^p_{L}({\mathbb R}^n),$  and so  \eqref{11d}  follows readily.
The proof of Theorem~\ref{th1.1} is complete.
 \end{proof}

 \bigskip

 \noindent
{\bf Acknowledgments.} P. Chen is supported by NNSF of China (Grant No. 12171489). J. Xiao is supported by the Natural Science Foundation of Guangdong Province (Grant No. 2019A1515010955).  The authors would like to  thank Lixin Yan for helpful  discussions.

\end{document}